\documentclass[12pt,oneside]{amsart}
\usepackage{amssymb}
\usepackage{amsfonts}
\usepackage{amscd}
\usepackage[matrix, arrow, curve]{xy}
\usepackage{graphicx}
\usepackage{color}

\numberwithin{equation}{section}
\newtheorem{theorem}{Theorem}%[section]

\newtheorem{corollary}[theorem]{Corollary}
\newtheorem{lemma}[theorem]{Lemma}
\theoremstyle{definition}

\theoremstyle{remark}
\newtheorem*{remark}{Remark}

\newtheorem*{conjecture}{Conjecture}

\counterwithout{equation}{section} % сквозная нумерация формул, без указания главы и
\begin{document}
%%%%%%%%%%%%%%%%%%%%%%%%%%%%%%%%%%%%%%%%%%%%%%%%%
%%%%%%%%%%%%  macrodefinitions
%%%%%%%%%%%%%%%%%%%%%%%%%%%%%%%%%%%%%%%%%%%%%%%%%
%  Macros (general)
%%%%%%%%%%%%%%%%%%%%%%%
%\newcommand{\MgNekp}{\mathcal{M}_{g,N+1}^{(k,p)}} %% moduli space
%\newcommand{\M}{\mathcal{M}_{g,N+1}^{(1)}}
\newcommand{\M}{\mathcal{M}}
\newcommand{\F}{\mathcal{F}}

\newcommand{\Teich}{\mathcal{T}_{g,N+1}^{(1)}}
\newcommand{\T}{\mathrm{T}}
%%%%   temporary
\newcommand{\corr}{\bf}
\newcommand{\vac}{|0\rangle}
\newcommand{\Ga}{\Gamma}
\newcommand{\new}{\bf}
\newcommand{\define}{\def}
\newcommand{\redefine}{\def}
\newcommand{\Cal}[1]{\mathcal{#1}}
\renewcommand{\frak}[1]{\mathfrak{{#1}}}
\newcommand{\Hom}{\rm{Hom}\,}
%%%%%%%%%%%%%%%%%%%%%%%%%%%%%%%%%%%%
%   Referencing Scheme of Martin
%%%%%%%%%%%%%%%%%%%%%%%%%%
\newcommand{\refE}[1]{(\ref{E:#1})}
\newcommand{\refCh}[1]{Chapter~\ref{Ch:#1}}
\newcommand{\refS}[1]{Section~\ref{S:#1}}
\newcommand{\refSS}[1]{Section~\ref{SS:#1}}
\newcommand{\refT}[1]{Theorem~\ref{T:#1}}
\newcommand{\refO}[1]{Observation~\ref{O:#1}}
\newcommand{\refP}[1]{Proposition~\ref{P:#1}}
\newcommand{\refD}[1]{Definition~\ref{D:#1}}
\newcommand{\refC}[1]{Corollary~\ref{C:#1}}
\newcommand{\refL}[1]{Lemma~\ref{L:#1}}
\newcommand{\refEx}[1]{Example~\ref{Ex:#1}}
\newcommand{\ovl}[1]{\overline{#1}}
\newcommand{\til}[1]{\widetilde{#1}}
\newcommand{\what}[1]{\widehat{#1}}
%%%%%%%%%%%%%%%%%%%%%%%%%%%%%%%%%%
\newcommand{\R}{\ensuremath{\mathbb{R}}}
\newcommand{\C}{\ensuremath{\mathbb{C}}}
\newcommand{\N}{\ensuremath{\mathbb{N}}}
\newcommand{\Q}{\ensuremath{\mathbb{Q}}}
\renewcommand{\P}{\ensuremath{\mathcal{P}}}
\newcommand{\Z}{\ensuremath{\mathbb{Z}}}
%%%%%%%%%%%%%%%%%%%%%%%%%%%%%%%%%%%%%%%%%%
\newcommand{\kv}{{k^{\vee}}}
%%%%%%%%%%%%%%%%%%%%%%%%%%%%%%%%%%%%%%%%%%%%%
\renewcommand{\l}{\lambda}
%%%%%%%%%%%%%%%%%%%%%%%%%%%%%%%%%%%%%%%%%%%%%%%%%%
\newcommand{\mft}{\mathfrak{t}}
\newcommand{\gb}{\overline{\mathfrak{g}}}
\newcommand{\dt}{\tilde d}     % Oleg
\newcommand{\hb}{\overline{\mathfrak{h}}}
\newcommand{\g}{\mathfrak{g}}
\newcommand{\h}{\mathfrak{h}}
\newcommand{\gh}{\widehat{\mathfrak{g}}}
\newcommand{\ghN}{\widehat{\mathfrak{g}_{(N)}}}
\newcommand{\gbN}{\overline{\mathfrak{g}_{(N)}}}
\newcommand{\tr}{\mathrm{tr}}
\newcommand{\gln}{\mathfrak{gl}(n)}
\newcommand{\son}{\mathfrak{so}(n)}
\newcommand{\spnn}{\mathfrak{sp}(2n)}
\newcommand{\sln}{\mathfrak{sl}}
\newcommand{\sn}{\mathfrak{s}}
\newcommand{\so}{\mathfrak{so}}
\newcommand{\spn}{\mathfrak{sp}}
\newcommand{\tsp}{\mathfrak{tsp}(2n)}
\newcommand{\gl}{\mathfrak{gl}}
\newcommand{\slnb}{{\overline{\mathfrak{sl}}}}
\newcommand{\snb}{{\overline{\mathfrak{s}}}}
\newcommand{\sob}{{\overline{\mathfrak{so}}}}
\newcommand{\spnb}{{\overline{\mathfrak{sp}}}}
\newcommand{\glb}{{\overline{\mathfrak{gl}}}}
\newcommand{\Hwft}{\mathcal{H}_{F,\tau}}
\newcommand{\Hwftm}{\mathcal{H}_{F,\tau}^{(m)}}

%%%%%%%%%%%%%%%%%%%%%%%%%%%%%%%%%%%%%%%%%%%%%%%%%%%%
\newcommand{\car}{{\mathfrak{h}}}    % Cartan subalgebra
\newcommand{\bor}{{\mathfrak{b}}}    % Borel subalgebra
\newcommand{\nil}{{\mathfrak{n}}}    % nilpotent subalgebra
\newcommand{\vp}{{\varphi}}
\newcommand{\bh}{\widehat{\mathfrak{b}}}  % Borel subalgebra of KN algebra
\newcommand{\bb}{\overline{\mathfrak{b}}}  % Borel subalgebra of KN algebra
\newcommand{\Vh}{\widehat{\mathcal V}}
\newcommand{\KZ}{Kniz\-hnik-Zamo\-lod\-chi\-kov}
\newcommand{\TUY}{Tsuchia, Ueno  and Yamada}
\newcommand{\KN} {Kri\-che\-ver-Novi\-kov}
\newcommand{\pN}{\ensuremath{(P_1,P_2,\ldots,P_N)}}
\newcommand{\xN}{\ensuremath{(\xi_1,\xi_2,\ldots,\xi_N)}}
\newcommand{\lN}{\ensuremath{(\lambda_1,\lambda_2,\ldots,\lambda_N)}}
\newcommand{\iN}{\ensuremath{1,\ldots, N}}
\newcommand{\iNf}{\ensuremath{1,\ldots, N,\infty}}

\newcommand{\tb}{\tilde \beta}
\newcommand{\tk}{\tilde \varkappa}
\newcommand{\ka}{\kappa}
\renewcommand{\k}{\varkappa}
\newcommand{\ce}{{c}}

\newcommand{\Pif} {P_{\infty}}
\newcommand{\Pinf} {P_{\infty}}
\newcommand{\PN}{\ensuremath{\{P_1,P_2,\ldots,P_N\}}}
\newcommand{\PNi}{\ensuremath{\{P_1,P_2,\ldots,P_N,P_\infty\}}}
\newcommand{\Fln}[1][n]{F_{#1}^\lambda}
\newcommand{\tang}{\mathrm{T}}
\newcommand{\Kl}[1][\lambda]{\can^{#1}}
\newcommand{\A}{\mathcal{A}}
\newcommand{\U}{\mathcal{U}}
\newcommand{\V}{\mathcal{V}}
\newcommand{\W}{\mathcal{W}}
\renewcommand{\O}{\mathcal{O}}
\newcommand{\Ae}{\widehat{\mathcal{A}}}
\newcommand{\Ah}{\widehat{\mathcal{A}}}
\newcommand{\La}{\mathcal{L}}
\newcommand{\Le}{\widehat{\mathcal{L}}}
\newcommand{\Lh}{\widehat{\mathcal{L}}}
\newcommand{\eh}{\widehat{e}}
\newcommand{\Da}{\mathcal{D}}
\newcommand{\kndual}[2]{\langle #1,#2\rangle}
\newcommand{\cins}{\frac 1{2\pi\mathrm{i}}\int_{C_S}}
\newcommand{\cinsl}{\frac 1{24\pi\mathrm{i}}\int_{C_S}}
\newcommand{\cinc}[1]{\frac 1{2\pi\mathrm{i}}\int_{#1}}
\newcommand{\cintl}[1]{\frac 1{24\pi\mathrm{i}}\int_{#1 }}
\newcommand{\w}{\omega}
\newcommand{\ord}{\operatorname{ord}}
\newcommand{\res}{\operatorname{res}}
\newcommand{\nord}[1]{:\mkern-5mu{#1}\mkern-5mu:}
\newcommand{\codim}{\operatorname{codim}}
\newcommand{\ad}{\operatorname{ad}}
\newcommand{\Ad}{\operatorname{Ad}}
\newcommand{\supp}{\operatorname{supp}}

%%%%%%%%%%%%%%%%%%%%%%%%%%%%%%%%%%%%%%%%%%%%%%%%
\newcommand{\Fn}[1][\lambda]{\mathcal{F}^{#1}}
\newcommand{\Fl}[1][\lambda]{\mathcal{F}^{#1}}
\renewcommand{\Re}{\mathrm{Re}}

\newcommand{\ha}{H^\alpha}

\define\ldot{\hskip 1pt.\hskip 1pt}
\define\ifft{\qquad\text{if and only if}\qquad}
\define\a{\alpha}
\redefine\d{\delta}
\define\w{\omega}
\define\ep{\epsilon}
\redefine\b{\beta} \redefine\t{\tau} \redefine\i{{\,\mathrm{i}}\,}
\define\ga{\gamma}
\define\cint #1{\frac 1{2\pi\i}\int_{C_{#1}}}
\define\cintta{\frac 1{2\pi\i}\int_{C_{\tau}}}
\define\cintt{\frac 1{2\pi\i}\oint_{C}}
\define\cinttp{\frac 1{2\pi\i}\int_{C_{\tau'}}}
\define\cinto{\frac 1{2\pi\i}\int_{C_{0}}}
%\define\cinttt{\frac 1{24\pi\i}\int_{C_{\tau}}}
\define\cinttt{\frac 1{24\pi\i}\int_C}
\define\cintd{\frac 1{(2\pi \i)^2}\iint\limits_{C_{\tau}\,C_{\tau'}}}
\define\dintd{\frac 1{(2\pi \i)^2}\iint\limits_{C\,C'}}
\define\cintdr{\frac 1{(2\pi \i)^3}\int_{C_{\tau}}\int_{C_{\tau'}}
\int_{C_{\tau''}}}
\define\im{\operatorname{Im}}
\define\re{\operatorname{Re}}
%\define\res{\text{res}}
\define\res{\operatorname{res}}
\redefine\deg{\operatornamewithlimits{deg}}
\define\ord{\operatorname{ord}}
\define\rank{\operatorname{rank}}
\define\fpz{\frac {d }{dz}}
\define\dzl{\,{dz}^\l}
\define\pfz#1{\frac {d#1}{dz}}

\define\K{\Cal K}
\define\U{\Cal U}
\redefine\O{\Cal O}
\define\He{\text{\rm H}^1}
\redefine\H{{\mathrm{H}}}
\define\Ho{\text{\rm H}^0}
\define\A{\Cal A}
\define\Do{\Cal D^{1}}
\define\Dh{\widehat{\mathcal{D}}^{1}}
\redefine\L{\Cal L}
\newcommand{\ND}{\ensuremath{\mathcal{N}^D}}
\redefine\D{\Cal D^{1}}
\define\KN {Kri\-che\-ver-Novi\-kov}
\define\Pif {{P_{\infty}}}
\define\Uif {{U_{\infty}}}
\define\Uifs {{U_{\infty}^*}}
\define\KM {Kac-Moody}
\define\Fln{\Cal F^\lambda_n}
%%%%%%%%%%%%%%%%%%%%
\define\gb{\overline{\mathfrak{ g}}}
\define\G{\overline{\mathfrak{ g}}}
\define\Gb{\overline{\mathfrak{ g}}}
\redefine\g{\mathfrak{ g}}
\define\Gh{\widehat{\mathfrak{ g}}}
\define\gh{\widehat{\mathfrak{ g}}}
%%%%%%%%%%%%%%%%%%%%%%%%%%
\define\Ah{\widehat{\Cal A}}
\define\Lh{\widehat{\Cal L}}
\define\Ugh{\Cal U(\Gh)}
\define\Xh{\hat X}
\define\Tld{...}
\define\iN{i=1,\ldots,N}
\define\iNi{i=1,\ldots,N,\infty}
\define\pN{p=1,\ldots,N}
\define\pNi{p=1,\ldots,N,\infty}
\define\de{\delta}

\define\kndual#1#2{\langle #1,#2\rangle}
\define \nord #1{:\mkern-5mu{#1}\mkern-5mu:}
%\define \MgN{{\Cal M}_{g,N}} %% moduli space
%\define \MgNp{{\Cal M}_{g,N}^{(p)}} %% moduli space
\newcommand{\MgN}{\mathcal{M}_{g,N}} %% moduli space
\newcommand{\MgNeki}{\mathcal{M}_{g,N+1}^{(k,\infty)}} %% moduli space
\newcommand{\MgNeei}{\mathcal{M}_{g,N+1}^{(1,\infty)}} %% moduli space
\newcommand{\MgNekp}{\mathcal{M}_{g,N+1}^{(k,p)}} %% moduli space
\newcommand{\MgNkp}{\mathcal{M}_{g,N}^{(k,p)}} %% moduli space
\newcommand{\MgNk}{\mathcal{M}_{g,N}^{(k)}} %% moduli space
\newcommand{\MgNekpp}{\mathcal{M}_{g,N+1}^{(k,p')}} %% moduli space
\newcommand{\MgNekkpp}{\mathcal{M}_{g,N+1}^{(k',p')}} %% moduli space
\newcommand{\MgNezp}{\mathcal{M}_{g,N+1}^{(0,p)}} %% moduli space
\newcommand{\MgNeep}{\mathcal{M}_{g,N+1}^{(1,p)}} %% moduli space
\newcommand{\MgNeee}{\mathcal{M}_{g,N+1}^{(1,1)}} %% moduli space
\newcommand{\MgNeez}{\mathcal{M}_{g,N+1}^{(1,0)}} %% moduli space
\newcommand{\MgNezz}{\mathcal{M}_{g,N+1}^{(0,0)}} %% moduli space
\newcommand{\MgNi}{\mathcal{M}_{g,N}^{\infty}} %% moduli space
\newcommand{\MgNe}{\mathcal{M}_{g,N+1}} %% moduli space
\newcommand{\MgNep}{\mathcal{M}_{g,N+1}^{(1)}} %% moduli space
\newcommand{\MgNp}{\mathcal{M}_{g,N}^{(1)}} %% moduli space
\newcommand{\Mgep}{\mathcal{M}_{g,1}^{(p)}} %% moduli space
\newcommand{\MegN}{\mathcal{M}_{g,N+1}^{(1)}} %% moduli space

%\define \mpt{(M,P_1,P_2,\ldots, P_N,\Pif)} %% moduli point
%\define \mpp{(M,P_1,P_2,\ldots, P_N)} %% moduli point
%\define \MgNn{{\Cal M}_{g,N}^{(1)}} %% moduli space
%\define \MgNen{{\Cal M}_{g,N+1}^{(1)}} %% moduli space
%\define \Mgo{{\Cal M}_{g,0}} %% moduli space
%\define \mptn{(M,P_1,P_2,\ldots, P_N,\Pif,z_1,\ldots,z_N,z_\infty)}
 %% moduli point
%\define \mppn{(M,P_1,P_2,\ldots, P_N,z_1,\ldots,z_N)} %% moduli point
\define \sinf{{\widehat{\sigma}}_\infty}
\define\Wt{\widetilde{W}}
\define\St{\widetilde{S}}
\newcommand{\SigmaT}{\widetilde{\Sigma}}
\newcommand{\hT}{\widetilde{\frak h}}
\define\Wn{W^{(1)}}
\define\Wtn{\widetilde{W}^{(1)}}
\define\btn{\tilde b^{(1)}}
\define\bt{\tilde b}
\define\bn{b^{(1)}}
\define \ainf{{\frak a}_\infty} %matrices with a finite number of
                                %diagonals

%
%%%%%%%%%% Olegs definitions %%%%%%%%%%%%%%%%%%%%%%%%%%%%%%%%%%%
\define\eps{\varepsilon}    % Oleg
\newcommand{\e}{\varepsilon}
\define\doint{({\frac 1{2\pi\i}})^2\oint\limits _{C_0}
       \oint\limits _{C_0}}                            % Oleg
\define\noint{ {\frac 1{2\pi\i}} \oint}   % Oleg
\define \fh{{\frak h}}     % Oleg
\define \fg{{\frak g}}     % Oleg
\define \GKN{{\Cal G}}   % affine Krichever-Novikov algebra % Oleg
\define \gaff{{\hat\frak g}}   % affine Krichever-Novikov algebra
\define\V{\Cal V}
\define \ms{{\Cal M}_{g,N}} %% moduli space
\define \mse{{\Cal M}_{g,N+1}} %% moduli space
%%%%%%%%%%%%%%%%%%%%%%%%%%%%%%%%%%%%%%
\define \tOmega{\Tilde\Omega}
\define \tw{\Tilde\omega}
\define \hw{\hat\omega}
\define \s{\sigma}
\define \car{{\frak h}}    % Cartan subalgebra
\define \bor{{\frak b}}    % Borel subalgebra
\define \nil{{\frak n}}    % nilpotent subalgebra
\define \vp{{\varphi}}
\define\bh{\widehat{\frak b}}  % Borel subalgebra of KN algebra
\define\bb{\overline{\frak b}}  % Borel subalgebra of KN algebra
\define\KZ{Knizhnik-Zamolodchikov}
\define\ai{{\alpha(i)}}
\define\ak{{\alpha(k)}}
\define\aj{{\alpha(j)}}
\newcommand{\calF}{{\mathcal F}}
\newcommand{\ferm}{{\mathcal F}^{\infty /2}}
\newcommand{\Aut}{\operatorname{Aut}}
\newcommand{\End}{\operatorname{End}}
%%%%%%%%%%%%%%%%%%%%%%%%%%%%%%%%%%%%%%%%%%%
%%%%%%%%%%%%%%%%  цвет %%%%%%%%%%%%%%%%%%%%%%%%%%%%%%%%%
\newcommand{\novoe}{{\color[rgb]{1,0,0}\bf (Новое)}}
\newcommand{\staroe}{{\color[rgb]{0,0,1}\bf (Старое)}}
\newcommand{\red}{\color[rgb]{1,0,0}}
\newcommand{\blue}{\color[rgb]{0,0,1}}
\newcommand{\viol}{\color[rgb]{1,0,1}}%%%%%%%%%%%%%%%%%%%%%%%%%%%%%%%%%%%%%%%%%%%%%%

%%%%%%%%%%%%%%%%%%%%%%%%%%%%%%%%%
%%%%%%%%%%%%%%   for laxcent
%%%%%%%%%%%%%%%%%%%%%%%%%%%%%%%%%%
\newcommand{\laxgl}{\overline{\mathfrak{gl}}}
\newcommand{\laxsl}{\overline{\mathfrak{sl}}}
\newcommand{\laxso}{\overline{\mathfrak{so}}}
\newcommand{\laxsp}{\overline{\mathfrak{sp}}}
\newcommand{\laxs}{\overline{\mathfrak{s}}}
\newcommand{\laxg}{\overline{\frak g}}
\newcommand{\bgl}{\laxgl(n)}
%%%%%%%%%%%%%%%%%%%%%%%%
\newcommand{\tX}{\widetilde{X}}
\newcommand{\tY}{\widetilde{Y}}
\newcommand{\tZ}{\widetilde{Z}}
%%%%%%%%%%%%%%%%%%%%%%%%%%%%%%%%%%%%%%%%%%
%%%%%%%%%%%%  END of macrodefinitions
%%%%%%%%%%%%%%%%%%%%%%%%%%%%%%%%%%%%%%%%%

%%%%%%%%%%%%%%%%%%%%%%%%%%%%%%%%%
%Top-Matter
%%%%%%%%%%%%%%%%%%%%%%%%%%%%%%%
%%%%%%%%%%%%%%%%%    private header  %%%%%%%%%%%%%%%%%%%%

%\large{
%\title[]{Об одном представлении решений 2-мерного конечнозонного потенциального уравнения Шредингера}
\title[2-dimensional finite-gap Schr\"{o}dinger operator]{2-dimensional finite-gap Schr\"{o}dinger operator and double coverings of curves with two or more ramification points}
\author[O.K.Sheinman]{O.K.Sheinman}
%\date{\today}
\thanks{}
\address{Steklov Mathematical Institute of the Russian Academy of Sciences}
\dedicatory{}
\maketitle
%\begin{abstract}
%Библиография: 22 названия.
%{\bf Ключевые слова:}
%\end{abstract}
%\tableofcontents
%%%%%%%%%%%%%%%%%%%%%%%%%%%%%%%%%%%%%%%%
Two-dimensional Schr\"{o}dinger operators that are finite-gap operators at a single energy level are introduced in \cite{Dubr_Krich_Nov}. In \cite{Nov_Ves_1,Nov_Ves_4} the potentiality conditions for them have been studied, that are conditions for the magnetic term to be absend. Besides their physical importance, these works played a crucial role in solving out the Riemann--Shottki problem of indetification of Prymians of smooth double coverings with two ramification points in the class of principally polarized Abelian varieties \cite{Taiman_Mat_sb,Taiman_UMN,Krich_Prym_1,Krich-Grush_DUKE}.

For smooth coverings with more than two branch points the Prym varieties are no longer principally polarized. In \cite{Fay} certain isogenic to them principally polarized varieties are introduced, referred to as isoPrymians here. In the present work, the potentiality conditions are interpreted in terms of isoPrymians, and a related approach to identification of isoPrymians of smooth double coverings of curves with two or more ramification points is proposed. In particular, on this way we obtain a representation of the Baker--Akhieser function via Prym theta functions with characteristics, and the corresponding expression for the potential of the Schr\"{o}dinger operator.
%%%%%%%%%%%%%%%%%%%%%%%%%%%%%%%%%%%%%%%%%%%%%%%%%%%%%%
\subsection*{Abel--Prym transform. IsoPrymian}

Let  $\Sigma$ be a  genus $g$ smooth projective curve over $\C$, $\w_1,\ldots,\w_g$ be a base of normalized holomorphic differentials: $\oint_{a_j}\w_i=2\pi{\rm i}\d_{ij}$, $\w:=(\w_1,\ldots,\w_g)^T$. By $Div_d(\Sigma)$, we denote the collection of degree $d$ divisors on $\Sigma$.

The mapping $Ab: Div_d(\Sigma)\to \C^d$ defined by the relation $Ab(\zeta_1+\ldots+\zeta_d)=\sum_{j=1}^d\int^{\zeta_j}\w$, where
$\zeta_1,\ldots , \zeta_d\in\Sigma$, is referred to as \emph{Abel transform} (we assume that the initial point of integration (the base point) is the same for all integrals but don't specify it). The transformation $Ab$ is well-defined as a mapping to the \emph{Jacobian} $Jac(\Sigma)$ of the Riemann surface $\Sigma$, where
\[
 Jac(\Sigma)=\C^g/\Z(2\pi{\rm i}E,B),\quad B=\left(\oint_{b_j}\w_i\right)^{i=1,\ldots,g}_{j=1,\ldots,g},
\]
 $\Z(2\pi{\rm i}E,B)$ is the lattice generated by the columns of the matrices $2\pi{\rm i}E$ and~$B$.

For $d=g$ the mapping $Ab$ is almost everywhere invertible on $Jac(\Sigma)$, and $Ab^{-1}$ is referred to as the  \emph{Jacobi invertion map}. The following theorem takes place:
\begin{theorem}[Abel--Jacobi theorem] The Abel transform $Ab: Div_g(\Sigma)\to Jac(\Sigma)$, and the inverse Jacobi map establish a birational equivalence $Div_g(\Sigma)\simeq Jac(\Sigma)$.
\end{theorem}

$Jac(\Sigma)$ is an Abelian variety. By an Abelian variety we mean a complex projective torus here, and by a principally polarized Abelian variety we mean such torus represented as a quotient of a complex vector space by a maximal rank lattice.

%%%%%%%%%%%%%%%%%%%%%%%%%%%%%%%%%%%%%%%%%%%%%%%%%%%%%%%%%%%%%%%%%

Let $\Sigma$ be supplied with a holomorphic involution $\s:\Sigma\to\Sigma$.  Then it is related with a covering $\Sigma\to\Sigma_\s$ where $\Sigma_\s=\Sigma/\s$. Let $g_\s=genus(\Sigma_\s)$. A meromorphic 1-form $\w$ on $\Sigma$ is called a Prym differential if $\s^*\w=-\w$.
Let $h$ denote the dimension of the space of holomorphic Prym differentials. There is always a $\s$-invariant base of fundamental cycles $\{ a_\a, b_\a | \a=1,\ldots, g_\s \}\cup\{ a_j,b_j | j=g_{\s+1},\ldots,h\}$ on $\Sigma$ such that
\[
  a_\a+\s(a_{\a+h})=b_\a+\s(b_{\a+h}) = 0,\quad a_j+\s(a_j)=b_j+\s(b_j)=0.
\]
We define the normalized Prym differentials  with respect to this base: these are holomorphic differentials $\w_\a$, $\w_j$  such that  $\oint_{a_\b}\w_\a=2\pi{\rm i}\d_{\a,\b} $, $\oint_{a_j}\w_i=2\pi{\rm i}\d_{i,j} $, $\oint_{a_j}\w_\a=\oint_{a_\a}\w_i=0$ (where $\b=1,\ldots, g_\s $, $i=g_{\s+1},\ldots,h$).

%\vskip-10pt
The Abel--Prym map is defined as follows: $\displaystyle A(\zeta_1+\ldots+\zeta_d)=\sum_{k=1}^d\int^{\zeta_k}\begin{pmatrix}
                                         \w_\a \\
                                         \w_j \\
                                       \end{pmatrix}
$ where $\begin{pmatrix}
\w_\a \\
\w_j \\
\end{pmatrix}$ is a short notation for  $(\w_1,\ldots,\w_{g_\s},\w_{g_\s+1},\ldots,\w_h)^T$.                         The Prym matrix is defined by the following relation \cite{Fay}:
{
\renewcommand{\arraystretch}{2}
%\extrarowheight4pt
\begin{equation*}
\Pi =
\left(\begin{array}{c|c}
        \Pi_{\a\b}  &  \Pi_{\a j}    \\
        \hline
        \Pi_{i\b} &  \Pi_{ij}   \\
        \end{array}\right) =
\left(\begin{array}{c|c}
        \int_{b_\b}\w_\a &  \frac{1}{2}\int_{b_j}\w_\a    \\
        \hline
        \int_{b_\b}\w_i &  \frac{1}{2}\int_{b_j}\w_i   \\
        \end{array}\right),
\end{equation*}
}% \setstretch
and the Prym theta $\theta$-function by the relation
\[
  \theta(z,\Pi)=\sum_{N\in \Z^h}\exp(\frac{1}{2}N^T\Pi N+N^Tz),\quad z\in\C^h.
\]
It is the theta function of the Abelian variety
$P_0=\C^h/\Z(2\pi{\rm i}E,\Pi)$ which we call isoPrymian. The Abel--Prym map is well defined as a map $Div_{2h}(\Sigma)\to P_0$.

The Prym variety $P$ is defined as the Abelian subvariety in $Jac(\Sigma)$ which is the image of the alternation map $alt:\zeta\to Ab(\zeta)-Ab(\s\zeta)$, $\zeta\in Div(\Sigma)$. $P_0$~is an unramified $2^k$-sheet covering of~$P$ where $k=h-g_\s$.

%%%%%%%%%%%%%%%%%%%%%%%%%%%%%%%%%%%%%%%%
%%%%%%%%%%%%%%%%%%%%%%%%%%%%%%%%%%%%%%%%%%%%%%%

%%%%%%%%%%%%%%%%%%%%%%%%%%%%%%%%%%%%%%%%%%%%%%%%%%%%%%%%%%%%%%%%%
%\subsection*{Обращение Якоби для изопримианов}
%%%%%%%%%%%%%%%%%%%%%%%%%%%%%%%%%%%%%%%%%%%%%%%%%%%%%%%%%%%%%%%%%%
\subsection*{Veselov--Novikov conditions}
Assume the ramification divisor of the covering $\Sigma\to\Sigma/\s$ to be someway represented as $\sum_{j=0}^k (Q'_j+Q''_j)$. Let $\Delta=K+\sum_{j=0}^k (Q'_j+Q''_j)$ where $K$ is the canonical divisor on $\Sigma$. We call $\zeta\in Div_{2h}(\Sigma)$ a Veselov--Novikov divisor if
\begin{itemize}
  \item[(VN1)] $\zeta+\s\zeta\sim \Delta$ (where $\sim$ denotes the linear equivalence of divisors);
  \item[(VN2)] there exists a meromorphic differential $\Omega$ on $\Sigma$ with zero divisor $\zeta+\s\zeta$, and pole divisor $\sum_{j=0}^{k}(Q_j'+Q_j'')$, such that $\res_{Q_j'}\Omega+\res_{Q_j''}\Omega=0$.
\end{itemize}
We refer to the conditions (VN1) and (VN2) as Veselov--Novikov conditions. They are similar to the potentiality conditions for the two-dimensional  Shr\"{o}dinger operator obtained in \cite{Nov_Ves_4}. The condition (VN2) is obviously stronger than (VN1); according to \cite{Nov_Ves_4} it follows from the requirement that (VN1) takes place on the singular curve $\Sigma'=\Sigma/\{ Q'_j=Q''_j | j=1,\ldots,k \}$ obtained from $\Sigma$ by means of identification of the points $Q'_j$  and $Q''_j$ for $j=1,\ldots,k$. We don't use that. In the form we use here, the condition (VN2) is exactly what is used in  \cite{Nov_Ves_4} in the proof of potentiality (\cite[Theorem 1]{Nov_Ves_4}, Lemma \ref{L:fBA_ex} below). The condition (VN1), as we formulate it, emerges in \cite{Krich_Prym_1}.

We show that the condition (VN1) holds true for the divisors which are preimages of almost all points of the isoPrymian under the Abel--Prym map.

Let $\theta$ be the Prym theta function, $e\in \C^h$, $A$ be the Abel--Prym map, $Ab$ be the Abel map, $F_e(P)=\theta(A(P)-e)$.
\begin{theorem}[\cite{Fay}]\label{T:Fay}
If $F_e(P)$ is not equal to zero identically, $\zeta=div F_e$, then
\begin{itemize}
\item[$1^\circ$]
$\deg\zeta=2h$ ($h=\dim P_0=\dim Prym(\Sigma)$);
\item[$2^\circ$]
$Ab(\zeta)=\phi(e)+ Ab(\Delta')$ where $\Delta'=\sum_{j\ge 0}(Q_j'+Q_j'')+\pi^*\Delta_\s\in Jac(\Sigma)$, $2\Delta_\s\sim K_\s$,  and for $e=(e_\a,e_j)$ $\phi(e)=(e_\a,2e_j,e_\a)^T\in\C^g$;
\end{itemize}
\end{theorem}
\noindent
(here, as above, $(e_\a,e_j)=(e_1,\ldots,e_{g_\s},e_{g_\s+1},\ldots,e_h)$).
\begin{lemma}\label{L:cor2}
Assume $\zeta=div F_e$. Then $A(\zeta)=\eps e$ where $\eps(e_\a,e_j)=(e_\a,2e_j)$.
\end{lemma}
\begin{proof}
As it is shown in \cite{Sh_Matsb_25}, an analog of the statement $2^\circ$ of the theorem \ref{T:Fay} holds for the Abel--Prym map: $A(\zeta)=\eps e+\Delta''$ where $\Delta''$  does not depend on~$\zeta$.

Obviously, $\s\zeta$ is the zero divisor of the function $\theta(A(\s P)-e)$. Due to skew-symmetry of $A(P)$ with respect to $\s$, and to eveness of $\theta$, we have
\[
   \theta(A(\s P)-e)= \theta(-A(P)-e)=\theta(A(P)+e),
\]
which implies $\s\zeta=div F_{-e}$. For this reason $A(\s\zeta)=-\eps e+\Delta''$, hence $A(\zeta)+A(\s\zeta)=2\Delta''$. By skew-symmetry of the Abel--Prym map with respect to the involution, $A(\zeta)+A(\s\zeta)=0$ which gives $\Delta''=0$.
\end{proof}
\begin{lemma}[\cite{Sh_TMPh_26}]\label{L:cor3}
There takes place a birational equivalence
\begin{equation}\label{E:bir}
   P_0\simeq \{\zeta\in Div_{2h} | Ab(\zeta)+Ab(\s\zeta)=Ab(\Delta)\}.
\end{equation}
\end{lemma}
There has not been found out the right hand side of \refE{bir} in \cite{Sh_TMPh_26}; we do it here. In course of proof of the previous lemma we obtained that $\s\zeta=div F_{-e}$. Hence it follows from \refT{Fay}.$2^\circ$ that
\[
  Ab(\zeta)=\phi(e)+ Ab(\Delta'),\ \ Ab(\s\zeta)=-\phi(e)+ Ab(\Delta'),
\]
which implies
\[
   Ab(\zeta)+ Ab(\s\zeta)=Ab(2\Delta')
\]
where $2\Delta'=2\sum_{j\ge 0}(Q_j'+Q_j'')+\pi^*(2\Delta_\s)$. Since $2\Delta_\s\sim  K_\s$ where $K_\s$ is the canonical divisor on $\Sigma_\s$, and $\pi^*K_\s+\sum_{j\ge 0}(Q_j'+Q_j'')\sim K$, we obtain \refE{bir}.
\begin{corollary}\label{C:cor1}
Assume $\zeta=div F_e$. Then the divisor $\zeta$ satisfies to the condition (VN1).
\end{corollary}
The corollary can be also easily derived from the relation
\[
   \pi_*\zeta=K_\s+\sum_{j=0}^{k} (\pi_*Q_j'+\pi_*Q_j'')
\]
obtained in course of proof of Lemma 5.6 \cite{Fay}. Indeed, pull back both parts of the equality to $\Sigma$, by applying  the $\pi^*$ to them. Since $\pi^*\pi_*\zeta=\zeta+\s\zeta$, $\pi^*\pi_*Q_j'=2Q_j'$,  $\pi^*\pi_*Q_j''=2Q_j''$ and $\pi^*K_\s+\sum_{j=0}^{k} (Q_j'+Q_j'')\sim K$, we obtain the required  condition (VN1).

\refT{Fay}, Lemmas \ref{L:cor2},\ref{L:cor3}, and Corollary \ref{C:cor1}  can be summarized as the following commutative diagram:
%\subsection*{Отображения абелевых многообразий}
\begin{equation*}
\xymatrix{
Div(\Sigma)\ar[rr]^{A}\ar[rrd]^{A} &  &  \C^h\ar[rr]^{\phi}_{\hookrightarrow}\ar[d]^{/\Z(2\pi{\rm i}E,\Pi)}  &   & \C^g\ar[ddd]^{/\Z(2\pi{\rm i}E,B)} \\
Div_{2h}(\Sigma)\ar[dd]_{Ab/(Ves-Nov)}\ar[u]^\hookrightarrow &  &P_0\ar[d]_{cover}\ar[ll]^{Jacobi\ inv} & &\\
& &P\ar[dll]^{+\Delta}  &  &\\
P+\Delta\ar[rrrr]_\hookrightarrow & &   &  & Jac(\Sigma)\ar[ull]_{alt}
}
\end{equation*}
%%%%%%%%%%%%%%%%%%%%%%%%%%%%%%%%%%%%%%%%

%%%%%%%%%%%%%%%%%%%%%%%%%%%%%%%%%%%%%%%%%%%%%%%%%%%%%%%%%%%%%%%%%
\subsection*{Baker--Akhieser function}
Let $\Sigma\to\Sigma_\s$ be a smooth double covering with $2(k+1)$ ramification points $Q'_0,Q''_0, Q'_1,Q''_1,\ldots, Q'_k,Q''_k$, $k\ge 0$.

We choose local parameters $w_1$ and $w_2$ in neighborhoods of the points $Q_0'$ and $Q_0''$, respectively, such that
\begin{equation}\label{E:loc_preobr}
  \s w_1=-w_1,\ \s w_2=-w_2.
\end{equation}
\begin{lemma}[\cite{Nov_Ves_4,Dubr_Krich_Nov}]\label{L:fBA_ex}
For any non-special positive degree $2h$ divisor $\zeta$ there exists a unique function $\psi$ on $\Sigma$ (called normalized Baker--Akhieser function) with the following analytic properties:
 \begin{equation*}\label{E:fBA_prop}
  \begin{aligned}
   1^\circ\quad & \psi\ \text{is meromorphic outside the points}\ Q'_0,Q''_0,\ \text{and its pole divisor is equal to}\ \zeta;   \\
   2^\circ\quad & \psi(w_1)= \exp({\rm i}zw_1^{-1})(1+\xi_1(z,\ovl{z})w_1+O(w_1^{2}))\ \text{for}\ w_1\to 0;                                \\
   3^\circ\quad & \psi(w_2)= c(z,\ovl{z})\exp({\rm i}zw_2^{-1})(1+\xi_2(z,\ovl{z})w_2+O(w_2^{2}))\ \text{for}\ w_2\to 0;\\
   4^\circ\quad & \psi(Q'_j)=\psi(Q''_j),\ j=1,\ldots,k.
 \end{aligned}
 \end{equation*}
If $\zeta$ is a Veselov--Novikov divisor then $\psi$ satisfies the $2D$ Schr\"{o}dinger equation \newline $(\partial{\ovl\partial}+u)\psi=0$ with $u=-{\ovl\partial}\ln\xi_1$ where $\partial=\partial/\partial z$, $\ovl\partial=\partial/\partial {\ovl z}$.
\end{lemma}
\begin{remark}
The first version of the Lemma has been formulated in \cite{Dubr_Krich_Nov} in the case of the divisor $\zeta$ of degree $g$. The divisor of degree $2h$ and the relations $4^\circ$ emerged in \cite{Nov_Ves_4}. In this form, the relation $u=-{\ovl\partial}\xi_1$ is written down in \cite{Taiman_Mat_sb}.
\end{remark}

Let $\Omega_1$ be a differential on $\Sigma$ with the main part $d(w_1^{-1})$ at $Q'_0$, $\Omega_2$ be a differential with the main part $d(w_2^{-1})$ at $Q''_0$, both having zero $a$-periods. By relations \refE{loc_preobr}, both $\Omega_1$ and $\Omega_2$ are Prym differentials.
Let $\theta$ be the Prym theta function. Consider the Prym theta function $\what{\theta}_j$, $0\le j\le k$, with characteristics $\b$ where
\[\what{\theta}_j = \theta\!\begin{bmatrix}
                          0 \\
                          \b_j \\
                        \end{bmatrix}, \quad
\b_j=(0,\ldots,0,1/2,\underbrace{0,\ldots,0}_{j-1\ \text{times}})^T,\ \theta\!\begin{bmatrix}
                          0 \\
                          \b \\
                        \end{bmatrix}(z):=\theta(z+\b).
\]
It is assumed that ${\what\theta}_0=\theta$.
\begin{lemma}\label{L:fBA_expl2}
The following functions are well-defined and linear independent for $j=0,1,\ldots,k$:
\begin{equation}\label{E:fBA_expl2}
  \psi_j(P,z,\ovl{z})=\frac{{\what\theta}_j(A(P)+U_1z+U_2\ovl{z}-e)\theta(e)} {\theta(A(P)-e){\what\theta}_j(U_1z+U_2\ovl{z}-e)}\cdot \exp\left(z\int^P\Omega_1+\ovl{z}\int^P\Omega_2\right),
\end{equation}
where $e=A(\zeta)$.
\end{lemma}
\begin{proof} $1^\circ$. The question of well-definiteness descends to the invariance of the right hand side of \refE{fBA_expl2} with respect to a change of the integration path (which is assumed to be the same in the definition of $A(P)$, and in the exponent). The last question, in turn, descends to the consideration of the quasiperiodicity multipliers of theta functions involved in~\refE{fBA_expl2}. The latter are nontrivial only under adding of some $b$-periods to the integration path. In that case $A(P)$ receives an addition of the form $\Pi M$, $M\in\Z^h$, which leads to the emergence of the quasiperiodicity multiplier
\[
   \exp(-\frac{1}{2}M^T\Pi M-M^T(A(P)+U_1z+U_2\ovl{z}-e))\exp(-2\pi{\rm i}M^T\b_j)
\]
in the nominator. Its part containing $U_1z+U_2\ovl{z}$ is compensated by the integrals in the exponent, and the remainder by $\theta(A(P)-e)$ in denominator, except for $\exp(-2\pi{\rm i}M^T\b_j)$.

Consider the additional quasiperiodicity multiplier $\exp(-2\pi{\rm i}M^T\b_j)$. When adding the cycle $b_m$ to the integration path, where $m\le h-j$, there emerges the argument shift in $\what{\theta}_j$ by the column $\Pi_m$ of the matrix $\Pi$, hence  $M=(0,\ldots,0,1,0\ldots,0)^T$ ($1$ at $m$'s position). Thus for $m\le h-j$   $M^T\b_j=0$. For $h-j+1\le m\le h$ we have $m>h-k$, and the shift of argument is equal to $2\Pi_m$, by definition of the Prym matrix, hence $M$ is an even vector. In both cases $\exp(-2\pi{\rm i}M^T\b_j)=1$, and the additional quasiperiodicity multiplier is equal to~1.

$2^\circ$. Suppose that the functions $\psi_j$ are linear dependent. It is equivalent to a linear dependence of non-exponential parts of the functions $\psi_j$, because the have the same the exponential part. Denote these non-exponential parts by $\phi_j$, then for some $\l_j$ (not all of which are equal to zero) $\sum_{j=0}^{k}\l_j\phi_j=0$. Lets shift the argument by $\Pi_m$, $h-k+1\le m\le h$ in this equality. By the same considerations as above we shall obtain that all coefficients $\l_j$ remain the same, except for $\l_{j_m}$, $j_m=h-m+1$, which changes its sign. Indeed, $M^T\b_{j_m}=1/2$, hence $\exp(-2\pi{\rm i}M^T\b_{j_m})=-1$. We obtained that $\sum_{j<j_m}\l_j\phi_j -\l_{j_m}\phi_{j_m}+ \sum_{j_m<j\le k}\l_j\phi_j =0$, which together with $\sum_{j=0}^{k}\l_j\phi_j=0$ implies $\l_{j_m}=0$. Obviously, for $m=h-k+1,\ldots,h$\ \ $j_m$ takes all values between $k$ and 1. Then also $\l_0=0$.
\end{proof}
\begin{corollary}
The functions $\psi_j$, $j=0,1,\ldots,k$ form a base in the space of Baker --Akhieser functions with the pole divisor $\zeta$, and given exponentional part.
\end{corollary}
Indeed, by the Riemann--Roch theorem the dimension of the space of Baker --Akhieser functions with the pole divisor $\zeta$, and given exponentional part is equal to $2h-g+1=k+1$.
\begin{lemma}[\cite{Nov_Ves_4}]\label{L:fBA_expl1}
The normalized Baker--Akhieser function introduced in \refL{fBA_ex} admits a unique representation in the form $\psi=\sum_{j=0}^k c_j\psi_j$ where $c_j$ satisfy the system of equations
\begin{equation}\label{E:lin_ur}
  \sum_{j=0}^{k}c_j=1; \quad \sum_{j=0}^{k}c_j\psi_j(Q_s')=\sum_{j=0}^{k}c_j\psi_j(Q_s''),\ s=1,\ldots,k.
\end{equation}
\end{lemma}
Existence of such representation of $\psi$ as of a linear combination of the functions $\psi_j$ follows from the previous corollary, and from the fact that the pole divisorof $\psi$ is equal to $\zeta$ by definition. The uniqueness follows from the sysnem of equations \refE{lin_ur}.
\begin{remark}
In terms of the Prym theta function, the Baker--Akhieser function first had been constructed in \cite{Nov_Ves_1}, in the case of double coverings with two ramification points, i.e. for $k=0$. In that case our representation for  $\psi$ descends to the one given in \cite{Nov_Ves_1}. If $\theta$ would be the Riemann theta function, the expressions $\psi_j$ would not be well-defined functions on the curve, for the reason that the idea of redoubling columns as in the proof of \refL{fBA_expl2} is no longer working. In this case, there is the Krichever's construction for a base in the space of Baker--Akhieser functions with a given pole divisor $\zeta$. In frame of that construction the pole divisors of $\psi_j$ are degree $g$ subdivisors of the divisor $\zeta$. However, that base could not be used in order to formulate a conjecture of identification of isoPrymians below, because it looks impossible to interpret the property of being a subdivisor in the language of Abelian varieties.
\end{remark}
%%%%%%%%%%%%%%%%%%%%%%%%%%%%%%%%%%%%%%%%%%%
\begin{lemma}\label{L:poten}
For the potential in \refL{fBA_ex}, the following expression takes place:
\begin{equation}\label{E:poten}
   u=-\frac{\partial^2}{\partial z\partial{\ovl z}} \sum_{j=0}^{k}c_j\ln{\what\theta}_j(U_1z+U_2{\ovl z}-e).
\end{equation}
\end{lemma}
\begin{proof}
According to \refL{fBA_ex} $u=-{\ovl\partial}\xi_1$ where $\xi_1$ is a coefficient of the expansion $\psi(w_1)=\exp(zw_1^{-1})(1+\xi_1(z,\ovl{z})w_1+O(w_1^{2}))$ when $w_1\to 0$, i.e. when $P\to Q_0'$. Introduce expansions of the functions $\psi_j$: $\psi_j(w_1)=\exp(zw_1^{-1})(1+\xi_1^j(z,\ovl{z})w_1+O(w_1^{2}))$. It follows from \refL{fBA_expl1} that $\xi_1=\sum_{j=0}^{k}c_j\xi_1^j$. Obviously, $u=-\sum_{j=0}^{k}c_j{\ovl\partial}\xi_1^j$.

To calculate $\xi_1^j$ we use the scheme which is standard in frame of the theory of Baker--Akhieser functions. Expand the non-exponential part of the expression \refE{fBA_expl2} (denote it $\Theta_j$, i.e. $\psi_j=\Theta_je^{z\int^P\Omega_1+\ovl{z}\int^P\Omega_2}$) in powers of~$w_1$. To this purpose, observe that $A(P)=\int_{Q_0'}^{P}\w$ where $\w$ is a column of basis holomorphic Prym differentials. Introducing the notation $\w=\varphi(w_1)dw_1$ we obtain $A(P)=\varphi(0)w_1+O(w_1^2)$, and $A(Q_0')=0$. By the Tailor formula
\begin{align*}
  {\what\theta}_j(A&(P)+ U_1z+U_2\ovl{z}-e)= \\
  = & {\what\theta}_j(U_1z+U_2\ovl{z}-e) + \varphi(0)^T\frac{\partial}{\partial w}{\what\theta}_j(U_1z+U_2\ovl{z}-e) w_1 +O(w_1^2) = \\
  = &  {\what\theta}_j(U_1z+U_2\ovl{z}-e)\left(1+ \varphi(0)^T\frac{\partial}{\partial w}\ln{\what\theta}_j(U_1z+U_2\ovl{z}-e) w_1+O(w_1^2)  \right)
\end{align*}
where $w$ is the full set of arguments of the theta function. Similarly
\[
  {\theta}(A(P)-e)=
  {\theta}(e)\left(1+ \varphi(0)^T\frac{\partial}{\partial w}\ln{\theta}(e) w_1+O(w_1^2)  \right).
\]
Plugging to \refE{fBA_expl2}, we obtain
\begin{align*}
  \Theta_j &= \frac{1+ \varphi(0)^T\frac{\partial}{\partial w}\ln{\what\theta}_j(U_1z+U_2\ovl{z}-e) w_1+O(w_1^2)} {1+ \varphi(0)^T\frac{\partial}{\partial w}\ln{\theta}(e) w_1+O(w_1^2)}=\\
  &= 1+ \varphi(0)^T\left( \frac{\partial}{\partial w} \ln{\what\theta}_j(U_1z+U_2\ovl{z}-e)+ \frac{\partial}{\partial w}\ln{\theta}(e)\right) w_1+O(w_1^2).
\end{align*}
Thus
\begin{align*}
  & \xi_1^j = \varphi(0)^T\left(\frac{\partial}{\partial w}\ln{\what\theta}_j(U_1z+U_2\ovl{z}-e)+ \frac{\partial}{\partial w}\ln{\theta}(e)\right), \\
  & {\ovl\partial}\xi_1^j = {\ovl\partial} \varphi(0)^T\frac{\partial}{\partial w}\ln{\what\theta}_j(U_1z+U_2\ovl{z}-e),
\end{align*}
since $\frac{\partial}{\partial w}\ln{\theta}(e)$ does not depend on  $z,\ovl z$. By the relation between $b$-periods of a single-pole differential, and evaluations of of the corresponding normalized holomorphic differentials at the point of its pole, we have $\phi(0)=\oint_b\Omega_1$, and then, by definition, $\oint_b\Omega_1=U_1$. Hence $\phi(0)^T\frac{\partial}{\partial w}=U_1^T\frac{\partial}{\partial w}=\frac{\partial}{\partial z}$. Finalyy, we obtain
\[
  {\ovl\partial}\xi_1^j= \partial{\ovl\partial}\ln{\what\theta}_j(U_1z+U_2\ovl{z}-e)
\]
which immediately implies the statement of the lemma.
\end{proof}
\begin{remark}
For $k=0$ the expression \refE{poten} coincides with the expression for the potential given in \cite{Nov_Ves_1}.
\end{remark}
 %%%%%%%%%%%%%%%%%%%%%%%%%%%%%%%%%%%%%%%%%%%%%%%%%%%%%

\subsection*{Characterization of isoPrymians}
Due to the connection between the Veselov--Novikov conditions and isoPrymians, established by Lemmas \ref{L:cor2},\ref{L:cor3} and Corollary  \ref{C:cor1}, and due to the representation of the Baker--Akhieser function given by Lemmas \ref{L:fBA_expl2}, \ref{L:fBA_expl1}, we arrive to the following conjecture similar to the Krichever's theorem~\cite{Krich_Prym_1}.
\begin{conjecture}
An indecomposable principally polarized Abelian variety $X,\theta$ of dimension $h$ is the isoPrymian of a smooth double covering of a curve, with $k+1$ pair of ramification points if, and only if the function $\psi=\sum_{j=0}^{k}c_j\psi_j$ where
\[
  \psi_j(z,{\ovl z})=\frac{{\what\theta}_j(A+U_1z+U_2\ovl{z}+Z)} {{\what\theta}_j(U_1z+U_2\ovl{z}+Z)}\cdot e^{p_1z+p_2{\ovl z}},
\]
satisfies to the equation
$\displaystyle ({\partial\ovl{\partial}}+u)\psi=0$ with $u$ given by the relation \refE{poten}, for some constant $A,U_1,U_2\in\C^h$, $p_1,p_2,c_0,c_1,\ldots,c_k\in\C$, and arbitrary $Z\in X$.
\end{conjecture}
Lemmas \ref{L:fBA_ex}--\ref{L:poten} prove our conjecture in direction "only if". For $k=0$ it descends to the Krichever's theorem~\cite{Krich_Prym_1}.

%%%%%%%%%%%%%%%%%%%%%%%%%%%%%%%%%%%%%%%%%%%%%%%%
The author is grateful to A.P.Veselov and I.A.Taimanov for discussing the initial versions of the work.

%%%%%%%%%%%%%%%%%%%%%%%%%%%%%%%%%%%%%%%%%%%%%%%%%%%%%%%%%%%%%%%%%
\bibliographystyle{amsalpha}

\begin{thebibliography}{A}
%%%%%%%%%%%%%%%%%%%%%%%%%%%%%%%%%%%%%%%%
\define\PL{Phys. Lett. B}
\define\NP{Nucl. Phys. B}
\define\LMP{Lett. Math. Phys. }
\define\JGP{JGP}
\redefine\CMP{Commun. Math. Phys. }
\define\JMP{J.  Math. Phys. }
\define\Izv{Math. USSR Izv. }
\define\FA{Funct. Anal. and Appl.}
\def\Pnas{Proc. Natl. Acad. Sci. USA}
\def\PAMS{Proc. Amer. Math. Soc.}
\def\UMNr{Russ. Math. Surv.}

%%%%%%%%%%%%%%%%%%%%%%%%%%%%%%%
%%%%%%%%%%%%%%%%%%%%%%%%%%%%%%%
\bibitem{Dubr_Krich_Nov}
%Б. А. Дубровин, И. М. Кричевер, С. П. Новиков. \emph{Уравнение Шредингера в периодическом поле и римановы поверхности}. Докл. АН СССР, 229:1 (1976), 15--18.
%
B.~A.~Dubrovin, I.~M.~Krichever, S.~P.~Novikov. \emph{The Schr\"odinger equation in a periodic field and Riemann surfaces}. Dokl. Akad. Nauk SSSR, 1976, Vol. 229, issue 1, p. 15--18.

\bibitem{Nov_Ves_1}
%А. П. Веселов, С. П. Новиков. \emph{Конечнозонные двумерные потенциальные операторы Шредингера. Явные формулы и эволюционные уравнения}, Докл. АН СССР, 1984, том 279, номер 1,  20--24.
%
A.~P.~Veselov, S.~P.~Novikov. \emph{Finite-gap two-dimensional potential Schrodinger operators. Explicit formulas and evolution equations}.
Dokl. Akad. Nauk SSSR, 1984, Vol. 279, issue 1, p. 20--24.

\bibitem{Nov_Ves_4}
% А.П. Веселов,  С.П. Новиков. \emph{Конечнозонные двумерные операторы Шредингера. Потенциальные операторы}, Докл. АН СССР, 1984, том 279, номер 4, 784--788.

Veselov A.P., Novikov S.P. \emph{Finite-gap two-dimensional Schr\"{o}dinger} operators. Potential operators. Doklady, 1984, Vol. 279, issue 4, p. 784--788.

\bibitem{Taiman_Mat_sb}
%И.А.Тайманов. \emph{Многообразия Прима разветвлённых накрытий и нелинейные уравнения.} Матем. сб., 1990, том 181, № 7, стр. 934--950.

I.~A.~Taimanov. \emph{Prym varieties of branched coverings and nonlinear equations}. Math. USSR-Sb., 1991, Vol. 70, issue 2, p. 367--384.

\bibitem{Taiman_UMN}
%И.А.Тайманов. \emph{Секущие абелевых многообразий, тэта-функции и солитонные уравнения.} УМН, 1997, том 52, выпуск 1(313), стр. 149--224.
%
I.~A.~Taimanov. \emph{Secants of Abelian varieties, theta functions, and soliton equations}. Russian Math. Surveys, 1997, Vol. 52, issue 1,
p. 147--218.

\bibitem{Krich_Prym_1}
I. Krichever. \emph{A characterization of Prym varieties}. International Mathematics Research Notices, Volume 2006, Article ID 81476, DOI: 10.1155/IMRN/2006/81476,
arXiv:math/0506238.



\bibitem{Krich-Grush_DUKE}
S.Grushevsky, I.Krichever
\emph{Integrable  discrete Schr\"{o}dinger equations and a characterization of Prym varieties by a pair of quadrisecants}.
Duke Mathematical Journal, Vol. 152, No. 2, p. 317--371.

\bibitem{Krich_Met_alg_geom}
%И. М. Кричевер. \emph{Методы алгебраической геометрии в теории нелинейных уравнений}. УМН, 32:6(198) (1977), 183--208.
%
I.~M.~Krichever. \emph{Methods of algebraic geometry in the theory of non-linear equations}. Russian Math. Surveys, 1977, Vol. 32, issue 6,
p. 185--213.
%
\bibitem{Fay}
J.D.Fay. \emph{Theta-functions on Riemann surfaces}. Lecture notes in mathematics, Vol. 352, Springer--Verlag, 1973.


\bibitem{Sh_Matsb_25}
%О.К.Шейнман. \emph{Обращение преобразования Абеля--Прима при наличии дополнительной инволюции}. Математический сборник, том 216 (2025), № 12, 125--144.
%
O.~K.~Sheinman. \emph{Inversion of the Abel--Prym map in presence of an additional involution}. Sb. Math., 2025, Vol. 216, issue 12, p. 1754--1772.


\bibitem{Sh_TMPh_26}
%О.К. Шейнман. \emph{Обращение Якоби и решения систем Хитчина}. Теоретическая и математическая физика, Том 226, № 3, 2026, стр. 452--464.
%
O.~K.~Sheinman. \emph{Jacobi inversion and solutions of Hitchin systems}. Theor Math Phys, 2026, Vol. 226, p. 393--403. https://doi.org/10.1134/S0040577926030025



%%%%%%%%%%%%%%%%%%%%%% %%%%%%%%%%%%%%%% %%%%%%%%%%%%%% %%%%%%%%%%%%

%%%%%%%%%%%%%%%%%%%%%% %%%%%%%%%%%%%%%% %%%%%%%%%%%%%% %%%%%%%%%%%%
%%%%%%%%%%%%%%%%%%%%%%%%%%%%%%%%%%%%%%%%%%%%%%%%%%%%
\end{thebibliography}

\end{document}
%%%%%%%%%%%%%%%%%%%%%%%%%%%%%%%%%%%%%%%%%%%%%%%%
%%   THE END
%%%%%%%%%%%%%%%%%%%%%%%%%%%%%%%%%%%%%%%%%%%